\documentclass[12pt]{article}

\usepackage{amsmath, amsthm, amssymb}
\usepackage{enumerate}
\usepackage{pdflscape}
\usepackage{caption}
\usepackage{breqn}

\usepackage{ifpdf}
\ifpdf
\usepackage[pdftex]{graphicx}
\else
\usepackage[dvips]{graphicx}
\fi
\usepackage{tikz-cd}
 \usetikzlibrary{arrows,backgrounds}
\usepackage[all]{xy}

\usepackage{multicol}

\usepackage{tocvsec2}
\usepackage{breqn}

\usepackage{bbm}

\input xy
\xyoption{all}

\usepackage[pdftex,plainpages=false,hypertexnames=false,pdfpagelabels]{hyperref}
\newcommand{\arxiv}[1]{\href{http://arxiv.org/abs/#1}{\tt arXiv:\nolinkurl{#1}}}
\newcommand{\arXiv}[1]{\href{http://arxiv.org/abs/#1}{\tt arXiv:\nolinkurl{#1}}}
\newcommand{\doi}[1]{\href{http://dx.doi.org/#1}{{\tt DOI:#1}}}

\newcommand{\googlebooks}[1]{(preview at \href{http://books.google.com/books?id=#1}{google books})}

\usepackage{xcolor}
\definecolor{dark-red}{rgb}{0.7,0.25,0.25}
\definecolor{dark-blue}{rgb}{0.15,0.15,0.55}
\definecolor{medium-blue}{rgb}{0,0,.8}
\definecolor{DarkGreen}{RGB}{0,150,0}
\definecolor{rho}{named}{red}
\hypersetup{
 colorlinks, linkcolor={purple},
 citecolor={medium-blue}, urlcolor={medium-blue}
}

\usepackage{longtable}
\usepackage{fullpage}
\theoremstyle{plain}
\newtheorem{thm}{Theorem}[section]
\newtheorem*{thm*}{Theorem}

\newtheorem{cor}[thm]{Corollary}

\newtheorem*{cor*}{Corollary}

\newtheorem*{conj*}{Conjecture}
\newtheorem{lem}[thm]{Lemma}

\newtheorem{prop}[thm]{Proposition}
\newtheorem{quest}[thm]{Question}
\newtheorem*{quest*}{Question}
\newtheorem*{claim*}{Claim}

\theoremstyle{definition}
\newtheorem{defn}[thm]{Definition}

\newtheorem{sub-ex}[thm]{Sub-Example}

\newtheorem*{rem*}{Remark}

\newcommand{\comment}[1]{}

\newcommand{\be}{\begin{enumerate}[label=(\arabic*)]}
\newcommand{\ee}{\end{enumerate}}

\newcommand{\Z}{\mathbb{Z}}
\newcommand{\Q}{\mathbb{Q}}

\newcommand{\C}{\mathbb{C}}

\def\semicolon{,}
\def\applytolist#1{
 \expandafter\def\csname multi#1\endcsname##1{
 \def\multiack{##1}\ifx\multiack\semicolon
 \def\next{\relax}
 \else
 \csname #1\endcsname{##1}
 \def\next{\csname multi#1\endcsname}
 \fi
 \next}
 \csname multi#1\endcsname}

\def\calc#1{\expandafter\def\csname c#1\endcsname{{\mathcal #1}}}
\applytolist{calc}QWERTYUIOPLKJHGFDSAZXCVBNM,
\def\bbc#1{\expandafter\def\csname bb#1\endcsname{{\mathbb #1}}}
\applytolist{bbc}QWERTYUIOPLKJHGFDSAZXCVBNM,
\def\bfc#1{\expandafter\def\csname bf#1\endcsname{{\mathbf #1}}}
\applytolist{bfc}QWERTYUIOPLKJHGFDSAZXCVBNM,
\def\sfc#1{\expandafter\def\csname s#1\endcsname{{\sf #1}}}
\applytolist{sfc}QWERTYUIOPLKJHGFDSAZXCVBNM,
\def\fc#1{\expandafter\def\csname f#1\endcsname{{\mathfrak #1}}}
\applytolist{fc}QWERTYUIOPLKJHGFDSAZXCVBNM,

\newcommand{\noshow}[1]{}

\def\sdprod{{\times\!\vrule height5pt depth0pt width0.4pt\,}}

\begin{document}

\title{Vanishing of categorical obstructions for permutation orbifolds}
\author{Terry Gannon and Corey Jones}
\date{}
\maketitle

\begin{abstract}
The orbifold construction $A\mapsto A^G$ for a finite group $G$ is fundamental in rational conformal field theory. The construction of $Rep(A^G)$ from $Rep(A)$ on the categorical level, often called gauging, is also prominent in the study of topological phases of matter. Given a non-degenerate braided fusion category $\mathcal{C}$ with a $G$-action, the key step in this construction is to find a braided $G$-crossed extension compatible with the action. The extension theory of Etingof-Nikshych-Ostrik gives two obstructions for this problem, $o_3\in H^3(G)$ and $o_4\in H^4(G)$ for certain coefficients, the latter  depending on a categorical lifting of the action and is notoriously difficult to compute. We show that in the case where $G\le S_n$ acts by permutations on $\mathcal{C}^{\boxtimes n}$, both of these obstructions vanish. This verifies a conjecture of M\"uger, and constitutes a nontrivial test of the conjecture that all modular tensor categories come from vertex operator algebras or conformal nets.

\end{abstract}

\section{Introduction.}

Modular tensor categories play an important role in low dimensional physics, both in 2-dimensional conformal field theories and topological phases of matter. In conformal field theory, modular tensor categories arise as local representations of fields of observables. An important construction in this context is the \textit{orbifold construction}, which takes the fixed-point fields of observables by a global action of a finite group. In the completely-rational conformal net axiomatization, M{\"u}ger has shown that the resulting modular tensor category can be constructed from the original by a categorical procedure known as \textit{gauging} \cite{MR2183964}. Gauging of modular categories also appears in the context of topological phases of matter, describing the process of promoting a group of global symmetries to local symmetries \cite{1410.4540}. 

In the setting of abstract modular tensor categories, gauging is a multi-step process \cite{MR3555361}. The starting point is an action of $G$ on $\cC$ compatible with the braiding, i.e. a homomorphism $\rho:G\rightarrow EqBr(\cC)$, where $EqBr(\cC)$ is the group of equivalence classes of braided autoequivalences of $\cC$. To gauge $\cC$ by $G$, we first need to find a compatible \textit{braided G-crossed extension} of $\cC$. We then equivariantize by the associated action of $G$, which produces a new modular tensor category called a \textit{gauging of $\cC$ by $G$}. Gaugings may not exist at all, and if they do they are usually not unique. 

The problem is that associated to the initial homomorphism $\rho:G\rightarrow EqBr(\cC)$ are two cohomological obstructions to finding a compatible braided $G$-crossed extension \cite{MR2677836}. First one asks if $\rho$ can be lifted to a categorical action $\underline{\rho}: \underline{G}\rightarrow \underline{EqBr}(\cC)$ (we define $\underline{G},\underline{EqBr}(\cC)$ etc in Section 2 below). There is an obstruction to obtaining a lifting, given by a cohomology class $o_{3}(\rho)\in H^{3}(G, \operatorname{Inv}(\cC))$, where $\operatorname{Inv}(\cC)$ is the group of invertible objects in $\cC$, also known as the \textit{simple currents} of $\cC$. If $o_3(\rho)$ vanishes, the liftings form a torsor over $H^{2}(G, \operatorname{Inv}(\cC))$. Once we have a categorical action $\underline{\rho}:\underline{G}\rightarrow \underline{EqBr}(\cC)\cong \underline{Pic}(\cC)$, in order to obtain a braided $G$-crossed extension, we now need to lift this to a tri-functor $\underline{\underline{G}}\rightarrow\underline{\underline{Pic}}(\cC)$ \cite{MR2677836}, Theorem 7.12. The obstruction for this lifting is a cohomology class $o_{4}(\underline{\rho})\in H^{4}(G, \C^{\times})$. If this vanishes, the liftings form a torsor over $H^{3}(G,\C^{\times})$. Thus starting with $\rho$, there are two obstructions to the existence of a desired braided $G$-crossed extension of $\cC$, $o_3(\rho)$ and $o_4(\underline{\rho})$, where the latter depends on a choice of lifting of $\rho$.

 This obstruction theory has direct relevance to any (chiral) rational conformal field theory $A$, which we can axiomatize either as a strongly-rational vertex operator algebra or a completely-rational conformal net of von Neumann algebras. The following statements are theorems for conformal nets, though many are still only conjectures for vertex operator algebras. The category of representations $Rep(A)$ is a modular tensor category. Suppose a finite group $G$ acts faithfully on $A$. Then $G$ acts as automorphisms on $Rep(A)$ which respect the braiding. The orbifold $A^G$ (the subtheory of $A$ consisting of $G$-fixed-points) will also be rational, so $Rep(A^G)$ will also be a modular tensor category. This category contains $Rep\,G$ as a fusion subcategory --- this fact goes back to Doplicher-Haag-Roberts, but also follows because extending $A^G$ by the objects in that subcategory recovers $A$. The $G$-twisted (solotonic) representations of $A$ form the braided $G$-crossed category $Rep(A^G)\sdprod({Rep}\,G)$, called the de-equivariantization of $Rep(A^{G})$ \cite{MR2183964}. The equivariantization $({Rep}(A^G)\sdprod({Rep}\,G))^G$ recovers $Rep(A^G)$. In particular, the obstruction $o_3(\rho)$ must vanish, and for some resulting choice of $\underline{\rho}$ so must the obstruction $o_4(\underline{\rho})$. This means that if we can find some $G$-action $\rho$ on a modular tensor category $\cC$, for which either the obstruction $o_3(\rho)$ or all $o_4(\underline{\rho})$ fail to vanish, then no completely-rational conformal net can have a $G$-action which acts like $\rho$ (respectively $\underline{\rho}$) on its representations.

An obvious class of orbifolds are the \textit{permutation orbifolds}. Start with a rational conformal field theory $A$, say a conformal net. Write $\cC=Rep(A)$. Then any subgroup $G\le S_{n}$ acts by global automorphisms on the tensor power theory $A^{\otimes n}$; the subtheory $(A^{\otimes n})^G$ is called the permutation orbifold. Then $\rho_n$ acts on the objects and morphisms of the Deligne product $\cC^{\boxtimes n}$ in the obvious way by permuting factors, as given explicitly in Section \ref{PermutationActions}. As shown in Proposition \ref{o3vanish} below, this always lifts to the \textit{standard permutation categorical action} $\underline{\rho}_{n}$ of $G$ by permutation functors, which is a strict action whose tensorators are all identities. Hence the obstruction $o_3(\rho_n)$ must vanish. Even in this simple setting, however, the obstruction $o_4(\underline{\rho_n})$ is difficult to compute explicitly. If $o_4$ did not vanish for the standard permutation categorical action $\underline{\rho_n}$, then there could be no conformal net $A$ realizing $\cC$. But it has been conjectured that any modular tensor category is $Rep(A)$ for some rational conformal field theory $A$. For this reason, M{\"u}ger conjectured that $o_4(\underline{\rho_n})$ always vanishes (Conjecture 5.5, Appendix 5, \cite{MR2674592}). This same reason makes the question of whether all $o_4(\underline{\rho_n})$ vanish, both natural and compelling. The main result of our paper answers M{\"u}ger's conjecture affirmatively. 

\begin{thm}\label{MainThm} Let $G\le S_{n}$ be a finite group and $\cC$ a non-degenerate braided fusion category. Let $\underline{\rho}_{n}: \underline{G}\rightarrow \underline{EqBr}(\cC^{\boxtimes n})\cong \underline{Pic}(\cC^{\boxtimes n})$ be the standard permutation categorical action. Then $o_{4}(\underline{\rho})$ vanishes.
\end{thm}

Note that we are able to work in the \emph{a priori} greater generality of non-degenerate braided fusion categories, since our arguments don't require modular data. In Question \ref{question} we ask whether $o_4$ vanishes for the 2-cocycle twistings of $\underline{\rho}_n$. We wish to point out two immediate corollaries of our main result. By \cite{MR2677836}, Theorem 7.12, we obtain the following statement which explicitly addresses M{\"u}ger's conjecture.

\begin{cor} Let $\cC$ be a non-degenerate braided fusion category and $G\le S^{n}$. Then there exist braided $G$-crossed fusion categories whose trivially-graded component is $\cC^{\boxtimes n}$, and whose categorical action restricted to the trivial component is the standard permutation braided categorical action of $G$ on $\cC^{\boxtimes n}$. The equivalence classes (up to braided equivalence) of such extensions form a torsor over $H^{3}(G, \C^{\times})$.
\end{cor}

We point out M{\"u}ger actually conjectures that there is a \textit{distinguished} such category. Although we have a distinguished categorical action $\underline{\rho}_{n}: \underline{G}\rightarrow \underline{EqBr}(\cC^{\boxtimes n})\cong \underline{Pic}(\cC^{\boxtimes n})$, our result does not pick out a distinguished extension, and indeed we have several: the equivalence classes form a torsor as stated in the corollary, with no special class distinguished. M\"uger's Conjecture 6.3 discusses the permutation orbifolds of conformal field theory, but even in this context there is not a distinguished category: e.g. in the special case where $Rep(A)=\mathrm{Vec}_\bbC$ and 3 divides the order of $G$, exactly 3 different categories can arise (the resulting category depends on the central charge modulo 24 --- this has been observed by Marcel Bischoff, as well as Example 2.1.1 in \cite{1707.08388} and Theorem 2 in \cite{EGrecon}. Thus the ``distinguished'' part of M{\"u}ger's conjecture fails, though for conformal field theories the resulting categories certainly are severely limited. 

Given a finite $G$-set, \cite{MR2806495} construct a weak $G$-equivariant fusion category.  It is tempting to guess that this matches a braided $G$-crossed extension of Corollary 1.2  but we cannot say for sure: it is only weakly rigid, and not necessarily rigid, because it is constructed using modular functors, though it does give a fusion category in the $S_2$ case.

Bantay's \textit{Orbifold Covariance Principle} \cite{MR1782167} is the requirement that all properties of a rational conformal field theory have to be compatible with the fact that permutation orbifolds of a rational theory is rational. For example, you get fairly easily from this the congruence property of the SL$_2(\bbZ)$ representation (modular data) of these theories. Modulo one subtlety, Corollary 1.2 requires that a modular tensor category  must be compatible with the fact that its gaugings by permutation actions, which all must exist, are themselves all healthy modular tensor categories. This should have plenty of consequences, including a faster proof of the congruence property. The subtlety though is that \cite{MR2677836} ignores spherical  structures for extensions, and in this setting one expects the theory to be slightly different, though we expect the obstructions for permutation functors to be the same in any case.

If $\cC$ is a braided fusion category with M{\"u}ger center $Z_{2}(\cC)$, a \textit{minimal non-degenerate extension} is a non-degenerate braided fusion category $\cD$ containing $\cC$ as a full fusion subcategory such that $\text{FPdim}(\cD)=\text{FPdim}(\cC)\text{FPdim}(Z_{2}(\cC))$. These do not always exist (for example, see Proposition 4.11 of \cite{1712.07097}). However, the existence of such an extension for a category with $Z_{2}(\cC)\cong Rep\,G$ for some finite group $G$ is closely tied with the vanishing of the $o_{4}$ obstruction associated with the canonical categorical action of $G$ on the de-equivariantization of $\cC$. From the arguments of Appendix 5, Theorem 5.4 of \cite{MR2674592}, or more directly from Theorem 4.8 (i) of \cite{1712.07097}, we have the following immediate corollary of Theorem \ref{MainThm}.

\begin{cor} Let $\cC$ be a non-degenerate braided fusion category, $G\le S^{n}$, and $(\cC^{\boxtimes n})^{G}$ the equivariantization of $\cC^{\boxtimes n}$ by the associated categorical permutation action. Then there exists a minimal non-degenerate extension of $(\cC^{\boxtimes n})^{G}$.
\end{cor}

We hope these results clear the path to understanding the categories resulting from permutation gauging. Recent progress has been made in the case of $S_{2}$, \cite{FusionPermutation}. The results there show that basic descriptions in the general case promises to be very complicated, or at best, difficult to prove using known techniques. On the other hand, a great deal is known about permutation orbifolds of conformal field theories \cite{MR1606821}, \cite{MR1620424}, \cite{MR2116735}, \cite{EGrecon}. We believe that comparing permutation orbifolds in conformal field theory and permutation gauging of modular categories could provide new clues for answering the important question: does every modular category arise from conformal field theory?

The outline of the paper is as follows. In Section $2$, we review some basics of braided $G$-crossed extension theory, and discuss the behavior of the $o_{4}$ obstruction under tensor product splitting. In Section 3 we discuss some lemmas regarding the cohomological restriction maps for subgroups of symmetric groups. In Section 4, we define the standard permutation categorical actions, and investigate possible twistings. Finally, we apply the group cohomology lemmas to prove Theorem \ref{MainThm}. We include an appendix with GAP computations.

\medskip

\textbf{Acknowledgements}. The authors would like to thank Michael M{\"u}ger, Andrew Schopieray  and Zhenghan Wang for helpful comments. This research was conducted while the first author was visiting the Australian National University. The first author was supported by an NSERC Discovery grant. The second author was supported by Discovery Projects ‘Subfactors and symmetries’ DP140100732 and ‘Low dimensional categories’ DP160103479 from the Australian Research Council.

\section{Braided $G$-crossed extension theory.}

Let $\cC$ be a braided fusion category over $\C$. For definitions and basic properties of braided fusion categories, see the textbook \cite{MR3242743}. In this section we will recall the basics of braided $G$-crossed extension theory. The original theory can be found in \cite{MR2677836}, while an explanation of general extension theory following the same approach can be found in \cite{1711.00645}. An overview of the theory with a specific emphasis on gauging can be found in \cite{MR3555361}.

Given a $\cC$-module category $\cM$, we can equip it with the structure of a $\cC$-bimodule category using the braiding (following the conventions in \cite{MR3107567}, Section 2.8). We say that $\cM$ is \textit{invertible} if it is invertible as a $\cC$-bimodule category. We define the tri-category $\underline{\underline{Pic}}(\cC)$ as follows:

\begin{itemize}
\item
There is a single $0$ morphism
\item
$1$-morphisms are invertible $\cC$-module categories, and 1-composition is given by relative tensor product (as bimodules).
\item
$2$-morphisms are module equivalences, with $1$-composition given by balanced tensor product of functors, and 2-composition by ordinary composition of functors
\item
$3$-morphisms are module-functor natural isomorphisms, with the obvious compositions.
\end{itemize}

This tri-category is also a categorical 2-group, since there is only one $0$-morphism and all morphisms are invertible. We can truncate once to form the monoidal category (bicategory with one object) $\underline{Pic}(\cC)$, where the 2-morphisms are equivalence classes of module equivalences. Truncating again, we obtain the group $Pic(\cC)$. The reason we are interested in this tri-category is that it controls braided $G$-crossed extensions of $\cC$.

Recall a \textit{braided $G$-crossed extension} of a braided fusion category $\cC$ is a $G$-graded fusion category $\cD:=\bigoplus_{g\in G} \cC_{g}$ with $\cC_{e}=\cC$, equipped with a categorical action of $G$ on $\cD$, and isomorphisms $c_{X,Y}:X\otimes Y\rightarrow g(Y)\otimes X$ for any $Y\in \cD$ and $X\in \cC_{g}$ called the $G$-crossed braiding subject to a family of compatibilities and coherences, see Definition 4.41 \cite{MR2609644}. The $G$-crossed braiding must restrict to the original braiding on $\cC_{e}=\cC$. The coherences ensure that the categorical action restricts to a braided categorical action on $\cC$.

When $\cC$ is modular (or more generally, non-degenerate braided), then equivariantizing a braided $G$-crossed extension by the $G$-action produces a new modular (non-degenerate braided) fusion category. The process of starting with $\cC$ equipped with a $G$-action, finding a braided $G$-crossed extension whose $G$-action restricts on the trivial component to the one we started with, and then equivariantizing is called \textit{gauging} \cite{MR3555361}. While we are ultimately interested in the non-degenerate fusion category produced by gauging, we have to first construct and classify the braided $G$-crossed extensions.

To do so, let $\underline{\underline{G}}$ be the tri-category with one object, whose 1-morphisms are given by group elements, and all higher morphisms are identities. The 1-truncation (which is a monoidal category) is denoted $\underline{G}$. The punchline of \cite{MR2677836}, Theorem 7.12 is that braided $G$-crossed extensions of $\cC$ are classified by tri-functors $\underline{\underline{G}}\rightarrow \underline{\underline{Pic}}(\cC)$. Thus to understand gauging, we must first understand some higher categorical algebra.

But we are left with the daunting question: how does one build a tri-functor $\underline{\underline{G}}\rightarrow \underline{\underline{Pic}}(\cC)$? One could try starting with the much less scary notion of a monoidal functor $\underline{G}\rightarrow \underline{Pic}(\cC)$. This turns out to be a good idea, since the monoidal category $\underline{Pic}(\cC)$ is actually well-studied in a different guise. Let $\underline{EqBr}(\cC)$ denote the monoidal category whose objects are braided monoidal autoequivalences of $\cC$ and whose morphisms are monoidal isomorphisms. There is a canonical monoidal functor $\underline{F}: \underline{Pic}(\cC)\rightarrow \underline{EqBr}(\cC)$, described as follows.

For a module category $\cM$, we have two tensor functors $\alpha_{\pm}:\cC\rightarrow \operatorname{End}_{\cC}(\cM)$, called $\alpha$-inductions. In both cases, the underlying endofunctor of $\cM$ associated to $\alpha_{\pm}(X)$ is just the ordinary left module action by $X$. However to define an $\operatorname{End}_{\cC}(\cM)$ structure on this endofunctor, we use the braiding and its reverse, respectively. If $\cM$ happens to be invertible as a module, then both $\alpha_{\pm}$ are braided equivalences, so there exists $\Theta_{\cM}\in EqBr(\cC)$ such that 

$$\alpha_{+}=\alpha_{-}\circ \Theta_{\cM}.$$

\noindent By \cite{MR2677836}, Theorem 5.4, the assignment $\cM\rightarrow \Theta_{M}$ naturally extends to a monoidal functor 

\begin{equation}\label{F}
\underline{F}:\underline{Pic}(\cC)\rightarrow \underline{EqBr}(\cC).
\end{equation}

 A major result that is particularly important for gauging is that $\underline{F}$ is an equivalence if $\cC$ is non-degenerate, and thus there is an inverse $\underline{F}^{-1}$. Although the inverse is explicit, a serious difficulty is that it is tricky to explicitly work out the tensorator of $\underline{F}^{-1}$ in enough detail to be useful. For details on all of the above, see \cite{MR2677836}, Section 5.4.

%%%%%%%%%%%%%%%%%%%%%%

In any case, when $\cC$ is non-degenerate, we see that we have a bijection between equivalence classes of monoidal functors $\underline{G}\rightarrow \underline{Pic}(\cC)$ and monoidal functors $\underline{G}\rightarrow \underline{EqBr}(\cC)$, simply by composing with the monoidal equivalence $\underline{F}$. The latter sort of monoidal functor has a designated name.

\begin{defn} Let $G$ be a finite group and $\cC$ a braided fusion category. A \textit{braided categorical action} of $G$ on $\cC$ is a monoidal functor $\underline{\rho}:\underline{G}\rightarrow \underline{EqBr}(\cC)$.
\end{defn}

Braided categorical actions appear naturally in many contexts, and there are many interesting and well-understood examples (see Section \ref{PermutationActions} for the examples that are the purpose of this paper). We often  abbreviate ``braided categorical action" to  ``categorical" action, and whether or not it is braided should be clear from context. We now know how to take a braided categorical action of $G$ on $\cC$ and produce a monoidal functor $\underline{G}\rightarrow \underline{Pic}(\cC)$. But where do we get braided categorical actions? 

In general, the idea is to start simply with a homomorphism $\rho: G\rightarrow EqBr(\cC)$. Then we work backwards, trying to lift $\rho$ to a monoidal functor $\underline{\rho}: \underline{G}\rightarrow \underline{EqBr}(\cC)\cong \underline{Pic}(\cC)$. We then try to lift this to a tri-functor $\underline{\underline{\rho}}:\underline{\underline{G}}\rightarrow \underline{\underline{Pic}}$, which is what we need for a braided G-crossed extension. However, at each stage of categorification, there is a cohomological obstruction to obtaining a lift. If this obstruction vanishes, we can categorify the morphism, and the different possibilities of liftings will be a torsor over a certain cohomology group (different at each stage). We explain this process starting from the bottom.

\subsection{First extensions.}\label{firstext}

Let us begin with a homomorphism $\rho:G\rightarrow EqBr(\cC)\cong Pic(\cC)$. The goal is to extend this to a monoidal functor $\underline{\rho}:\underline{G}\rightarrow \underline{Pic}(\cC)\cong \underline{EqBr}(\cC)$, where here we identify $\underline{Pic}(\cC)\cong \underline{EqBr}(\cC)$ via $F$ (see \eqref{F}).

The first step is to choose a representative invertible module $\cC_{g}$ for each $g\in G$, such that $[\cC_{g}]=\rho(g)\in Pic(\cC)$. Then we need (equivalence classes of) module functors 

$$M_{g,h}:\cC_{g}\boxtimes_{\cC} \cC_{h}\rightarrow \cC_{gh},$$

\noindent which give the tensorator for the monoidal functor. But for this to be monoidal, we need these to satisfy the relation 

\begin{equation}\label{quasitensor}
M_{gh,k}\circ (M_{g,h}\boxtimes_{\cC} Id_{k})\cong M_{g,hk}\circ (Id_{g}\boxtimes_{\cC} M_{h,k} )\,,
\end{equation} 

\noindent $h\in G$, where this isomorphism must be as module functors. Note that the $\cC$-module endofunctors of an invertible module category are given by multiplication by an invertible object in $\cC$, hence for \textit{any} choice of $M_{g,h}:\cC_{g}\boxtimes_{\cC} \cC_{h}\rightarrow \cC_{gh}$, we define 

$$T(\cC,M)_{g,h,k}:= M_{gh,k}\circ (M_{g,h}\boxtimes_{\cC} Id_{k}) \circ \left(M_{g,hk}\circ (Id_{g}\boxtimes_{\cC} M_{h,k} ) \right)^{-1}\in \operatorname{Inv}(\cC)$$

These assemble into a 3-cocycle $T(\cC,M)\in Z^{3}(G, \operatorname{Inv}(\cC))$, where $\operatorname{Inv}(\cC)$ is a $G$-module via $\rho$. We let $o_{3}(\rho)\in H^{3}(G, \operatorname{Inv}(\cC))$ be the cohomology class of $T(\cC,M)$ in $H^{3}(G, \operatorname{Inv}(\cC))$. While $T(\cC,M)$ may depend on the choices $M_{g,h}$ the corresponding class $o_{3}(\rho)$ only depends on the homomorphism $\rho$. The results of \cite{MR2677836} say that $T(\cC,M)\cong Id_{ghk}$ if and only if $o_{3}$ is trivial. Thus we can lift $\rho$ to a monoidal functor (or equivalently, a categorical action) $\underline{\rho}:\underline{G}\rightarrow \underline{Pic}(\cC)\cong \underline{EqBr}(\cC)$ if and only if $o_{3}(\rho)$ vanishes. If $o_{3}(\rho)$ does indeed vanish, then the set of equivalence classes of liftings $\underline{\rho^{\prime}}:\underline{G}\rightarrow \underline{Pic}(\cC)\cong \underline{EqBr}(\cC)$ which truncate to the given homomorphism $\rho$ form a torsor over the group $H^{2}(G, \operatorname{Inv}(\cC))$. 

Now, if we have taken the obstruction theory path and demonstrated the existence of a lifting by showing $o_{3}(\rho)$ is trivial, we know the liftings form a torsor over $H^{2}(G, \operatorname{Inv}(\cC))$, but we have no specified base point so we don't know how to explicitly identify the liftings with $H^{2}$-cohomology classes. This is necessary for computing the next obstruction. There is a way to remedy this situation and avoid explicitly computing the obstruction at the same time. The point is that if we directly construct a monoidal functor (equivalently, a categorical action) $\underline{\rho}: \underline{G}\rightarrow \underline{EqBr}(\cC)$ which truncates to $\rho$, then we have found a specific lifting. By the above discussion, this automatically implies the obstruction $o_{3}(\rho)$ vanishes, and simultaneously gives us a basepoint for our torsor. Indeed, constructing a categorical action is literally constructing a witness to the vanishing of $o_{3}(\rho)$. This is the route we take for permutation actions in Section $4$. 

%Given a $c$, the collection of choices of $M$ such that the class of $T(\cC,M)$ vanishes is a torsor over the group $H^{2}(G, Inv(\cC))$, where again, $Inv(\cC)$ is a $G$ module via $\pi$.

%Using the monoidal equivalence $\underline{F}: \underline{EqBr}(\cC)\rightarrow \underline{Pic(\cC)}$, we see that choosing $M$ such that the obstruction vanishes is equivalent (using $F$) extending the original homomorphism $\pi:G\rightarrow EqBr(\cC)$ to a monoidal functor $\underline{\pi}:\underline{G}\rightarrow \underline{Pic}(\cC)\cong \underline{EqBr}(\cC)$. A monoidal functor from $\underline{G}$ to $\underline{EqBr}$ is called a \textit{braided categorical action}, or just a categorical action if it is clear from context wether this action is braided.

%Sometimes (as in our case), we \textit{start} with a categorical action which we can describe completely. Using $\underline{F}^{-1}$, we obtain a monoidal functor $\underline{G}\rightarrow \EqBr(\cC)\cong \underline{Pic}(\cC)$. However, the structure of the monoidal functor $\underline{F}^{-1}$ is sufficiently complicated that it is difficult in general to explicitly describe the monoidal functor.

\subsection{Second extensions.}\label{secondext}

Now suppose we are given a monoidal functor $\underline{\rho}: \underline{G}\rightarrow \underline{Pic}(\cC)\cong \underline{EqBr}(\cC)$. Then we have invertible $\cC$-module categories $\cC_{g}$ and $\cC$-module functors $M_{f,g}$ satisfying \eqref{quasitensor}. However, to extend this to a tri-functor, we need to pick explicit $\cC$-module isomorphisms 

$$A_{g,h,k}: M_{gh,k}\circ (M_{g,h}\boxtimes_{\cC} Id_{k})\rightarrow M_{g,hk}\circ (Id_{g}\boxtimes_{\cC} M_{h,k} ) $$

\noindent These $A$ are meant to give us associators for the corresponding $G$-extension, hence must satisfy the pentagon equation (see \cite{EGNO}). Inverting one side of the pentagon equation, we obtain a module functor automorphism of the module functor

$$M_{fgh,k}\circ (M_{fg,h}\boxtimes_{\cC} Id_{k})\circ (M_{f,g}\boxtimes_{\cC} Id_{h}\boxtimes_{\cC} Id_{k}),$$

\noindent defined by

\begin{dmath}\label{associativity}
v(\cC,M,A)_{f,g,h,k}=(M_{f,g}\boxtimes_{\cC} Id_{h}\boxtimes_{\cC} Id_{k})A^{-1}_{fg,h,k}\circ (Id_{f}\boxtimes_{\cC} Id_{g}\boxtimes_{\cC} M_{h,k})A^{-1}_{f,g,hk} \circ M_{f,ghk}(id_{f}\boxtimes_{\cC} A_{g,h,k} )\circ A_{f,gh,k}(Id_{f}\boxtimes_{\cC} M_{g,h}\boxtimes_{\cC}Id_{k})\circ M_{fgh,k}(A_{f,g,h}\boxtimes_{\cC} id_{k} )
\end{dmath}

But every module automorphism of a module equivalence is a scalar times the identity, which is easily seen to satisfy the cocycle condition, and thus we may view

$$v(\cC,M,A) \in Z^{4}(G, \C^{\times})$$

\noindent We then define $o_{4}$ to be the image of $v(\cC,M,A)$ in $H^{4}(G, \C^{\times})$. It turns out (unsurprisingly) that the isomorphisms $A_{f,g,h}$ can be rescaled to satisfy the pentagon equations (and thus give us a genuine monoidal category) if and only if the class $o_{4}\in H^{4}(G, \C^{\times})$ is trivial. Furthermore, this cohomology class does not depend on the choice of $A_{f,g,h}$, and thus is canonically associated to a monoidal functor $\underline{G}\rightarrow \underline{Pic}(\cC)$. See \cite{MR2677836}, Section 8.6 for details.

In summary, to gauge by a braided categorical action $\underline{\rho}:\underline{G}\rightarrow \underline{EqBr}(\cC)\cong \underline{Pic}(\cC)$, we must first extend $\underline{\rho}$ to a tri-functor, in order to obtain a braided $G$-crossed extension of $\cC$. This amounts to checking that an obstruction $o_{4}(\underline{\rho})\in H^{4}(G, \C^{\times})$ vanishes. We emphasize here that $o_{4}$ in general is quite difficult to compute. 

%The more general extension theory for fusion categories (rather than braided $G$-crossed extension theory for braided fusion categories) follows along the same general lines, replacing $\underline{\underline{Pic}}(\cC)$ with the 3-category $\underline{\underline{BrPic}}(\cC)$. In this context, the meanings of these obstructions and torsors can be clarified by considering a nearly classical example: the extensions $Vec_{\tilde{\omega}}(K)$ by $G$ of $N$-graded vector spaces $Vec_\omega(N)$. This is worked out in detail in the Appendix of \cite{MR2677836}. 

%Here, $\omega\in H^3(N,\bbC^\times)$ and $\tilde{\omega}\in H^3(K,\bbC^\times)$ twist the associativity isomorphisms. 
%An action $\rho$ of $G$ on $Vec_\omega(N)$ involves amongst other things a homomorphism $\overline{\rho}:G\rightarrow Out(N)$, and the $o_3$ obstruction recovers the Eilenberg--MacLane obstruction in $H^3(G,Z(N))$ for the existence of a group extension $1\rightarrow N\rightarrow K\rightarrow G\rightarrow 1$. If that obstruction vanishes, then the possible group extensions $K$ form a torsor over $H^2(G,Z(N))$, and this is part of the $H^2$-torsor for $\underline{\rho}$. The $o_4$ obstruction (and the remainder of the $o_3$ one) concern the existence of a lift of $\omega\in H^3(N,\bbC^\times)$ to $\tilde{\omega}\in H^3(K,\bbC^\times)$. 
%The $H^3$-torsor here reflects the elementary fact that multiplying any particular solution $\tilde{\omega}\in H^3(K,\bbC^\times)$ by the lift (inflation) of any $\omega'\in H^3(G,\bbC^\times)$ will give another solution.

We now point out some basic properties of the $o_{4}$ obstruction that will be fundamental in our arguments. Suppose we have $\underline{\rho}: \underline{G}\rightarrow \underline{EqBr}(\cC)\cong \underline{Pic}(\cC)$ for a non-degenerate braided fusion category $\cC$. If $H\le G$ then we can restrict $\underline{\rho}$ to obtain a categorical action of $H$. Let $Res^{G}_{H}: H^{4}(\cdot , \C^{\times})$ denote the restriction map on cohomology. It follows immediately from the definitions that

\begin{equation}\label{restriction}
o_{4}(\underline{\rho}|_{\underline{H}}) =Res^{G}_{H}o_{4}(\underline{\rho}).
\end{equation} 

\noindent Now, suppose $\cD$ is another non-degenerate braided fusion category, and we have ${\underline{\pi}}:\underline{H}\rightarrow \underline{EqBr}(\cD)\cong {\underline{Pic}}(\cD)$. We can define the obvious \textit{product action} $\underline{\rho}\times \underline{\pi}: \underline{G\times H}\rightarrow \underline{EqBr}(\cC\boxtimes \cD)$ by 

$$(\rho\times\pi)(g\times h)(X\boxtimes Y):=\rho(g)(X)\boxtimes \pi(h)(Y)$$

\noindent with tensorator natural isomorphisms 

$$\mu^{\rho\times \pi}_{g_{1}\times h_{1}, g_{2}\times h_{2}}:=\mu^{\rho}_{g_1\times g_{2}}\boxtimes\mu^{\pi}_{h_1\times h_{2}} $$ 

\medskip

The defining properties of a braided categorical action are easy to verify. We have the following lemma.

\begin{lem}\label{product} If both $o_{4}(\underline{\rho})$ and $o_{4}(\underline{\pi})$ vanish, then $o_{4}(\underline{\rho}\times \underline{\pi})$ vanishes.
\end{lem}

\begin{proof}
By \cite{MR2677836} Theorem 7.12, the $o_{4}$ obstruction associated to a categorical action vanishes if and only if there exists a braided $G$-crossed extension whose categorical action on the trivially graded component is the one we started with. If $\cC\subseteq \cE$ is a braided $G$-crossed extension corresponding to $\underline{\rho}$ and $\cD\subseteq \cF$ is a braided $H$-crossed extension corresponding to $\underline{\pi}$, then $\cE\boxtimes \cF$ has a canonical structure of a braided $G\times H$-crossed extension of $\cC\boxtimes \cD$ which restricts to the product action $\underline{\rho}\times \underline{\pi}$ on the trivially graded component.

\end{proof}

Let $\underline{1}_{\cD}: \underline{1}\rightarrow \underline{EqBr}(\cD)$ be the trivial categorical action. We have the following easy lemma.

\begin{lem}\label{trivialproduct} $o_{4}(\underline{\rho}\times \underline{1}_{\cD})$ vanishes if and only if $o_{4}(\underline{\rho})$ vanishes
\end{lem}

\begin{proof}
One direction follows from Lemma \ref{product}. For the other, suppose $o_{4}(\underline{\rho}\times \underline{1}_{\cD})$ vanishes. Then there exists a braided $G$-crossed fusion category $\cE$ with $(\cE)_{e}=\cC\boxtimes \cD$, whose categorical action on the trivial component is $\underline{\rho}\times \underline{1}_{\cD}$. But then the equivariantization $\cE^{G}$ is non-degenerate and contains $\cD$ as a full subcategory, so $\cE^{G}=\cD^{\prime}\boxtimes \cD$. But $\cC^{G}\subseteq D^{\prime}$, and thus de-equivariantizing we see that $\cE\cong \cF\boxtimes \cD$, where $\cF$ is a braided $G$-crossed extension on $\cC$ which restricts to $\underline{\rho}$ on $\cC$, hence $o_{4}(\underline{\rho})$ vanishes. 
\end{proof}

We remark that a much stronger statement generalizing both the above lemmas is true, namely $o_{4}(\underline{\rho}\times \underline{\pi}_{\cD})$ vanishes if and only if both $o_{4}(\underline{\rho})$ and $o_{4}(\underline{\pi}_{\cD})$ vanish. However, our later arguments only require the above statements, and we prefer to give these since their arguments are elementary, only using the existence of braided $G$-crossed extensions rather than detailed knowledge of extension theory itself.

\section{Group cohomology lemmas.}

In this section, we discuss some lemmas relating to the group cohomology of the symmetric groups. First we will discuss some relevant facts concerning cohomology groups with trivial action on the coefficient modules.

Consider the short exact sequence of abelian groups

$$0\rightarrow \bbZ \buildrel j\over\hookrightarrow\bbC\buildrel e^{2\pi i z}\over\longrightarrow \C^{\times}\rightarrow 0.$$
where $\C$ is the additive group of complex numbers, $j:\Z\hookrightarrow \C$ is the usual inclusion. This induces a long exact sequence in cohomology 

\begin{equation}\label{Long2}
\cdots\rightarrow H^k(G,\bbZ)\buildrel j\over\rightarrow H^k(G,\C)\buildrel e^{2\pi i z}\over\longrightarrow
H^k(G,\C^{\times}) \buildrel\delta\over\rightarrow H^{k+1}(G,\bbZ)\rightarrow\cdots
\end{equation}

\noindent Since $\C$ is a $\Q$-vector space, $H^{n}(G,\C)=0$, and for every $n$, the long exact sequence gives us $0\rightarrow H^{n}(G,\C^{\times})\buildrel \delta \over\rightarrow H^{n+1}(G, \Z)\rightarrow 0$, and thus the connecting map $\delta$ furnishes a canonical isomorphism. Furthermore, naturality of the long exact sequence shows that the following diagrams commute for all $n$ and all subgroups $H\le G$:

\[ \begin{tikzcd}
H^{n}(G,\C^{\times}) \arrow{r}{\delta} \arrow[swap]{d}{Res^{G}_{H}} & H^{n+1}(G,\Z) \arrow{d}{Res^{G}_{H}} \\%
H^{n}(H,\C^{\times}) \arrow{r}{\delta}& H^{n+1}(H,\Z)
\end{tikzcd}
\]

\noindent This allows us to consider questions concerning $H^{4}(G, \C^{\times})$ and restrictions to subgroups by asking the same question for $H^{5}(G, \Z)$ and restriction, which are easily computable in computer algebra programs such as GAP. 

Recall the following deep result concerning the cohomology theory of symmetric groups, which is referred to as \textit{Nakaoka Stability}.

\begin{thm}\label{Stabilization}\cite[Corollary 6.7]{MR0112134}. For any abelian group $M$, the restriction map $$Res^{S_{n}}_{S_{2k}}:H^{k}(S_{n},M)\rightarrow H^{k}(S_{2k}, M)$$ is an isomorphism for all $n\ge 2k$.
\end{thm}

Specializing to $k=4$, we we immediately obtain the following.

\begin{cor}\label{StableCor} The restriction map $Res^{S_{n}}_{S_{8}}: H^{4}(S_{n},\C^{\times})\rightarrow H^{4}(S_{8}, \C^\times)$ is an isomorphism for all $n\ge 8$
\end{cor}

We find that $H^4(S_n,\C^{\times})$ equals $0$, $0$, $\bbZ_2$, $\bbZ_2$, $\bbZ_2\times \bbZ_2$, $\bbZ_2\times \bbZ_2$, and $\bbZ_2\times\bbZ_2\times \bbZ_2$ respectively, 
for $n=2,3,4,5,6,7,$ and $n \ge 8$ respectively.

Below are two lemmas which are central to our arguments in the next section. We will give two proofs for each. The first proof requires little additional knowledge about group cohomology, but relies on computations performed in the computer algebra program GAP, code for which is contained in the Appendix. In Subsection \ref{compfree}, we give alternate proofs of these lemmas which do not require a computer. These arguments use facts about the cohomology ring of the symmetric groups with coefficients in $\bbF_{2}$ and the Steenrod square map.

\begin{lem}\label{S4} Consider the subgroup $\langle (12), (34)\rangle\cong S_{2}\times S_{2}\subseteq S_{4}$. Then the restriction map $Res^{S_{4}}_{S_{2}\times S_{2}}:H^{4}(S_{4}, \C^{\times})\rightarrow H^{4}(S_2\times S_2, \C^{\times})$ is injective.
\end{lem}

\begin{proof}[Proof 1.] The GAP code appearing in the Appendix \ref{Appendix} shows that the image of the restriction map in is $\Z_{2}$, and since $H^{4}(S_{4}, \C^{\times})\cong \Z_{2}$, it must be injective.

\end{proof}

\begin{lem}\label{S8} Consider the subgroup of $S_{8}$ preserving $\{1,2,3,4\}$, which is isomorphic to $S_{4}\times S_{4}$. Then the restriction map $Res^{S_{8}}_{S_{4}\times S_{4}}:H^{4}(S_{8}, \C^{\times})\rightarrow H^{4}(S_{4}\times S_{4}, \C^{\times})$ is injective.
\end{lem}

\begin{proof}[Proof 1.] Since $H^{4}(S_{8}, \C^{\times})\cong \Z_{2}\times \Z_{2}\times \Z_{2}$ is 2-torsion, for any subgroup $G\le S_{8}$ with $[S_{8}:G]$ odd, we have $Res^{S_{8}}_{G}: H^{4}(S_{8}, \C^{\times})\rightarrow H^{4}(G, \C^{\times})$ is injective (this follows for example from \cite{MR672956}, Proposition 9.5, (ii)). Let $G=\langle(12),(34),(13)(24),(56),(78),(57)(68),(15)(26)(37)(48), (35)(46)\rangle$. This contains the Sylow 2-subgroup (generated by the first 7 elements in the list) so its index is odd, hence the restriction map is injective on $H^{4}(\cdot , \C^{\times})$.

Now, define the group $H =\langle(12),(34),(13)(24),(56),(78),(57)(68)\rangle \le S_{4}\times S_{4}$, which is isomorphic to $(S_{2}\wr S_{2})\times(S_{2}\wr S_{2})$. Computations in GAP (see Appendix \ref{Appendix}) show that $H^{4}(G,\C^{\times})\cong H^{5}(G, \Z)\cong \Z^{9}_{2}$, and $\operatorname{Image}(Res^{G}_{H}(H^{4}(G,\C^{\times})))\cong \Z^{9}_{2}$. Thus $Res^{G}_{H}$ is an isomorphism on $H^{4}(\ \cdot\ , \C^{\times})$. Therefore $Res^{S_{8}}_{H}=Res^{G}_{H}\circ Res^{S_{8}}_{G}$ is injective. But $Res^{S_{8}}_{H}=Res^{S_{4}\times S_{4}}_{H}\circ Res^{S_{8}}_{S_{4}\times S_{4}}$ being injective on $H^{4}(\ \cdot\ , \C^{\times})$ implies $Res^{S_{8}}_{S_{4}\times S_{4}}$ is injective on $H^{4}(\ \cdot\ , \C^{\times})$, as desired.

\end{proof}

\subsection{A computer-free approach.}\label{compfree}

The goal of this section is to provide proofs of Lemmas \ref{S4} and \ref{S8} that do not require computer computations. Similar arguments with more details included are given in Section 4.8 of \cite{EGrecon}.  Let $\bbF_{2}$ denote the field with two elements. The cohomology groups $H^k(S_n,\bbF_2)$ are much better studied than those of $H^k(S_n,\C^{\times})$, due the additional structure of a graded ring using the cup product. 

The short exact sequence
$$0\rightarrow \bbZ\buildrel \times 2\over\rightarrow\bbZ\buildrel \pi\over\rightarrow\bbF_2\rightarrow 0$$
 gives us the long exact sequence in cohomology:

\begin{equation}\label{Long2}
\cdots\rightarrow H^k(G,\bbZ)\buildrel\times 2\over\rightarrow H^k(G,\bbZ)\buildrel \pi\over\rightarrow
H^k(G,\bbF_2) \buildrel\beta\over\rightarrow H^{k+1}(G,\bbZ)\rightarrow\cdots
\end{equation}
where $\pi$ is reduction modulo 2. Equation \eqref{Long2} says that the order 2 classes in $H^{k+1}(G,\bbZ)$ can be identified with the image of $\beta$.

We will need to know the connecting homomorphisms $\beta$ more explicitly. This can be done through the $0\rightarrow\bbZ_2\rightarrow\bbZ_4\buildrel\times 2\over \rightarrow\bbZ_2\rightarrow 0$ sequence.
The corresponding connecting homomorphism $H^k(G,\bbF_2)\rightarrow H^{k+1}(G,\bbF_2)$
is also called the \textit{Steenrod square} $Sq^1$. It is a derivation, i.e.
$Sq^1(xy) =Sq^1(x)y+xSq^1(y)$ (where the product is the cup product), and $Sq^1(a)=a^2$ when $a$ has degree 1. $Sq^1$ is known explicitly in many concrete examples. What matters for us is that $Sq^1=\pi\circ\beta$.

We learned earlier that $H^5(S_n,\bbZ)\cong H^{4}(S_{n}, \C^{\times})$ has exponent 2 for all $n$, so the map $ H^{5}(S_{n}, \Z)\buildrel\times 2\over\rightarrow H^{5}(S_{n}, \Z)$ is the $0$ map. Equation \eqref{Long2} then says that $\pi: H^{5}(S_{n},\Z)\rightarrow H^{5}(S_{n}, \bbF_{2})$ is injective and $\beta: H^{4}(S_{n},\bbF_{2})\rightarrow H^{5}(S_{n}, \Z)$ is surjective. Thus the image $Sq^1(H^4(S_n,\bbF_2))=\pi \circ \beta$ can be lifted to an isomorphism

\begin{equation}\label{SqImage}
\pi:H^5(S_n,\bbZ)\buildrel\sim\over \rightarrow Sq^1(H^4(S_n,\bbF_2))
\end{equation}

%Note that since the Sylow 2-subgroup of $S_4$ is the wreath square $\bbZ_2\wr S_2\cong D_4$, transfer tells us that the restriction of $ H^4(S_4,\C^{\times})=\Z_2$ to $H^4(D_4,\C^{\times})=\Z_2\times \Z_2$ is an injection, since the index of a Sylow 2-subgroup is necessarily odd. 

Lemmas \ref{S4} and \ref{S8} can be proved by hand using facts about $H^k(S_n,\bbF_2)$ contained in \cite{MR1317096}, especially Chapter VI, together with supplementary information on $H^k(S_8,\bbF_2)$ in \cite{MR1060683}.

\medskip

\begin{proof}[Proof 2 of Lemma \ref{S4}.] The cohomology ring $H^*(S_4,\bbF_2)$ is $\bbF_2[\sigma_1,\sigma_2,c_3]/(\sigma_1c_3)$, 
where the subscript as always indicates degree. Hence $H^4(S_4,\bbF_2)\cong \bbZ_2^3$ is spanned by $\sigma_1^4, \sigma_1^2\sigma_2, \sigma_2^2$ . The action of the Steenrod square $Sq^1$ on the generators
%$Sq^1$ acting on the generators sends $\sigma_1$ to $\sigma_1^2$, $\sigma_2$ to $\sigma_1\sigma_2+c_3$, and $c_3$ to 0. Using $H^4(S_4,\bbZ)=\bbZ_3\times\bbZ_2\times\bbZ_4$, 
is given in Chapter VI of \cite{MR1317096}: it  sends $\sigma_1$ to $\sigma_1^2$, $\sigma_2$ to $\sigma_1\sigma_2+c_3$, and $c_3$ to 0. Hence we can
identify (through the injection $\pi$ as in Equation \eqref{SqImage}) the nontrivial class in $H^5(S_4,\bbZ)$ with $Sq^1(\sigma_1^2\sigma_2)=\sigma_1^3\sigma_2$. 
Restrictions to $H^*(S_2\times S_2,\bbF_2)$ of the generators are also given there: the image is the
polynomial algebra $\bbF_2[\sigma'_1,\sigma'_2]\subset H^*(S_2\times S_2,\bbF_2)$, and  $\sigma_1\mapsto
\sigma'_1,\sigma_2\mapsto\sigma'_2,c_3\mapsto0$. Hence
%Now, restriction from $\sigma_1,\sigma_2,c_3\in H^*(S_4,\bbF_2)$ to $H^*(S_2\times S_2,\bbF_2)$ is explicitly understood. Let $\xi$ denote the nontrivial character of $\bbZ_2\cong S_2$: i.e. $\xi(0)=1,\xi(1)=-1$. Then $\sigma_1$ restricts to $\sigma_1(a,b)=\xi(a)\xi(b)$ where $(a,b)\in \bbZ_2\times\bbZ_2$, $\sigma_2((a,b),(a',b'))=\xi(a)\xi(b')$, and $c_3$ is killed. The restriction to $H^*( S_2\times S_2,\bbF_2)$ lies in $H^*(\bbZ_2\times\bbZ_2,\bbF_2)^{S_2}=\bbF_2[\sigma_1,\sigma_2]$, a polynomial algebra (see Section III.4 of \cite{MR1317096}), so there are no identities satisfied by the restrictions $\sigma_1,\sigma_2$, so
 $\sigma_1^3\sigma_2$ restricts to $\sigma_1^{\prime\,3}\sigma'_2$. By naturality of the long exact sequence as before, $\pi$ intertwines $Res^{S_{4}}_{S_{2}\times S_{2}}$ for 
$H^{5}(S_{4},\bbF_{2})$ and $H^{5}(S_{4},\Z)$, and we are done.

% \[ \begin{tikzcd}
%H^{5}(S_{4},\Z) \arrow{r}{\pi} \arrow[swap]{d}{Res^{S_{4}}_{S_{2}\times S_{2}}} & H^{5}(S_{4},\bbF_{2}) \arrow{d}{Res^{S_{4}}_{S_{2}\times S_{2}}} \\%
%H^{5}(S_{2}\times S_{2},\Z) \arrow{r}{\pi}& H^{5}(S_{2}\times S_{2},\bbF_{2})
%\end{tikzcd}
%\]

%We have shown that $Res^{S_{4}}_{S_{2}\times S_{2}}\circ \pi$ is injective, which implies the result.

\end{proof}

The proof of the second lemma is similar so we will only sketch it.

\medskip

\begin{proof}[Proof 2 of Lemma \ref{S8}.] The commutative ring $H^*(S_8,\bbF_2)$ is $\bbF_2[\sigma_1,\sigma_2,\sigma_3,c_3,\sigma_4,x_5,d_6,d_7]/R$ where $R$ consists of relations in degree 6 and higher, so $H^4(S_8,\bbF_2)\cong \bbZ_2^6$. % is spanned by $\sigma_1^4,\sigma_1^2\sigma_2,\sigma_1\sigma_3,\sigma_1c_3,\sigma_2^2,\sigma_4$
%(again, $H^5(S_8,\bbF_2)$ is redundant here).
Equation \eqref{SqImage} identifies $H^5(S_8,\bbZ)$ with the $\bbF_2$-span of $ 
\sigma_1^3\sigma_2+\sigma_1^2\sigma_3, \sigma_1^2c_3, x_5+\sigma_1\sigma_4$.
%"mod 2 reduction" from $H^4(S_8,\bbZ)=\bbZ_3\times \bbZ_4\times \bbZ_2^2$ to $H^4(S_8,\bbF_2)=\bbZ_2^6$ has image $\bbZ_2^3$, so the all important Steenrod square $Sq^1:H^4(S_8,\bbF_2)\rightarrow H^5(S_8,\bbF_2)$ better have kernel $\bbZ_2^3$. Indeed, 
%$Sq^1$ takes $\sigma_1, \sigma_2, \sigma_3, c_3, \sigma_4$ to $\sigma_1^2,\sigma_1\sigma_2+\sigma_3+c_3,\sigma_1c_3+\sigma_1\sigma_3,0,x_5+\sigma_1\sigma_4$ respectively. We compute that the image of $Sq^1$ is the $\bbF_2$-span of $ \sigma_1^3\sigma_2+\sigma_1^2\sigma_3, \sigma_1^2c_3, x_5+\sigma_1\sigma_4$, which we then identify with $H^5(S_8,\bbZ)$.
Restriction to $S_4\times S_4$ is given in \cite{MR1060683} (see the proof of Theorem 3.2), and we find that 
%\begin{align}\sigma_1&\,\mapsto \sigma_1\otimes1+1\otimes \sigma_1\nonumber\\
%\sigma_2&\,\mapsto\sigma_2\otimes 1+\sigma_1\otimes\sigma_1+1\otimes\sigma_2\nonumber\\
%\sigma_3&\,\mapsto\sigma_2\otimes\sigma_1+\sigma_1\otimes\sigma_2\nonumber\\ 
%c_3&\,\mapsto c_3\otimes 1+1\otimes c_3\nonumber\\ 
%\sigma_4&\,\mapsto\sigma_2\otimes\sigma_2\nonumber\\ 
%x_5&\,\mapsto c_3\otimes\sigma_2+\sigma_2\otimes c_3\nonumber\end{align}
% where `$\otimes$' denotes the Kunneth cross-product of cocycles.
the restriction is 3-dimensional. %In particular, we see that the restriction of $\sigma_1^3\sigma_2+\sigma_1^2\sigma_3$ contains $\sigma_1^3\sigma_2\otimes 1$ (in fact it is only one which restricts non-trivially to $S_4\times 1$), the restriction of $\sigma_1^2c_3 $ is $c_3\otimes \sigma_1^2+\sigma_1^2\otimes c_3$ (this restricts non-trivially to $S_6\times 1$), and the restriction of $x_5+\sigma_1\sigma_4$ contains for instance $c_3\otimes \sigma_2$.
Intertwining $\pi$ with restrictions as in the last proof gives us the desired result.

\end{proof}

%\begin{proof} Applying $Hom(\cdot, \C^{\times})$, Nakaoka Stability gives us isomorphisms 

%$$Hom(H_{4}(S_{n}, \Z), \C^{\times})\cong Hom(H^{4}(S_{8},\Z) ,\C^{\times}).$$ 
%By the dual universal coefficients theorem, We have the natural short exact sequence 
%$$0\rightarrow Ext^{1}(H_{3}(G,\Z), \C^{\times})\rightarrow H^{4}(S_{n}, \C^{\times})\rightarrow Hom(H_{4}(S_{n},\Z), \C^{\times})\rightarrow 0$$ 
%As $\C^{\times}$ is injective, the second term is $0$, and thus we have natural isomorphisms 
%$$H^{4}(S_{n}, \C^{\times})\cong Hom(H_{4}(S_{n},\Z), \C^{\times})$$ 
%Thus composing these with the dual inclusion isomorphisms, we get natural isomorphisms $H^{4}(S_{n}, \C^{\times})\rightarrow H^{4}(S_{8}, \C^{\times})$, and since the Nakaoka isomorphisms are induced by inclusion, naturality and duality imply these isomoprhisms are induced by restriction.

%\end{proof}

\section{Permutation actions.}\label{PermutationActions}

Let $\cC$ be a non-degenerate braided fusion category, $n\ge 1$, and $S_{n}$ the permutation group on $n$ elements. Let $G\le S_{n}$ be any subgroup. We have a group homomorphism $\rho_{n}: G\rightarrow EqBr(\cC^{\boxtimes n})$ obtained simply by permuting the factors. As discussed at the end of Section \ref{firstext}, one way to show $o_{3}(\rho_{n})$ vanishes and construct a specific lifting $\underline{\rho}_{n}:\underline{G}\rightarrow \underline{Pic}(\cC)\cong \underline{EqBr}(\cC)$ at the same time is to explicitly construct a braided categorical action of $G\le S_{n}$ on $\cC^{\boxtimes n}$.

We proceed to define $$\underline{\rho}_{n}:\underline{G}\rightarrow \underline{Pic}(\cC)\cong \underline{EqBr}(\cC)$$

\noindent Let $\rho_{n}(g)$ be the endofunctor of $\cC^{\boxtimes n}$ that acts on objects via

$$ \rho_{n}(g)(X_{1}\boxtimes \cdots \boxtimes X_{n})=X_{g^{-1}(1)}\boxtimes \cdots \boxtimes X_{g^{-1}(n)}$$

\noindent and similarly on simple tensor morphisms 

$$ \rho_{n}(g)(f_{1}\boxtimes \cdots \boxtimes f_{n})=f_{g^{-1}(1)}\boxtimes \cdots \boxtimes f_{g^{-1}(n)} .$$

\noindent Let $X=X_{1}\boxtimes \cdots \boxtimes X_{n}$ and $Y=Y_{1}\boxtimes \cdots \boxtimes Y_{n}$ be objects in $\cC^{\boxtimes n}$. We define the monoidal functor structure maps $\psi^{g}_{X,Y}:\rho_{n}(g)(X)\otimes \rho_{n}(g)(Y)\rightarrow \rho_{n}(g)(X\otimes Y)$ by 

$$\psi^{g}_{X,Y}:=id_{X_{1}\otimes Y_{1}}\boxtimes \cdots \boxtimes id_{X_{n}\otimes Y_{n}}.$$

\noindent Here we identify $$\rho_{n}(g)(X)\otimes \rho_{n}(g)(Y)=(X_{g^{-1}(1)}\otimes Y_{g^{-1}(1)})\boxtimes \cdots \boxtimes (X_{g^{-1}(n)}\otimes Y_{g^{-1}(n)})$$ with 

$$\rho_{n}(g)(X\otimes Y)=(X_{g^{-1}(1)}\otimes Y_{g^{-1}(1)})\boxtimes \cdots \boxtimes (X_{g^{-1}(n)}\otimes Y_{g^{-1}(n)}).$$

%$$\psi^{g}_{X,Y}: \rho_{n}(g)(X_{1}\boxtimes \cdots \boxtimes X_{n})\otimes \rho_{n}(g) (Y_{1}\boxtimes \cdots \boxtimes Y_{n})$$
%$$=(X_{g^{-1}(1)}\boxtimes \cdots \boxtimes X_{g^{-1}(n)})\otimes (Y_{g^{-1}(1)}\boxtimes \cdots \boxtimes Y_{g^{-1}(n)})$$
%$$=(X_{g^{-1}(1)}\otimes Y_{g^{-1}(1)})\boxtimes \cdots \boxtimes (X_{g^{-1}(n)}\otimes Y_{g^{-1}(n)})$$
%$$\rightarrow \rho_{n}(g)\left( (X_1\otimes Y_{1})\boxtimes \cdots \boxtimes (X_{n}\otimes Y_{n})\right)$$ 
%$$=(X_{g^{-1}(1)}\otimes Y_{g^{-1}(1)})\boxtimes \cdots \boxtimes (X_{g^{-1}(n)}\otimes Y_{g^{-1}(n)})$$

\noindent It is clear from the construction that each $\rho_{n}(g)$ is a braided autoequivalence of $\cC^{\boxtimes n}$. Furthermore, we see that 

$$\rho_{n}(g)\circ \rho_{n}(h)(X)=X_{h^{-1}g^{-1}(1)}\boxtimes \cdots \boxtimes X_{h^{-1}g^{-1}(n)}=\rho_{n}(gh)(X)$$

Since our action is strict, we can define the tensorators 

$$\mu_{g,h}\in Nat(\rho_{n}(g)\circ \rho_{n}(h), \rho_{n}(gh)) = Nat(\rho_{n}(gh), \rho_{n}(gh))$$ 

$$\mu_{g,h}:=id_{\rho_{n}(gh)}$$

The strictness properties of our action make it easy to verify that this data assembles into a (strict) monoidal functor $\underline{\rho}_{n}:\underline{G}\rightarrow \underline{EqBr}(\cC^{\boxtimes n})\cong \underline{Pic}(\cC^{\boxtimes n})$ which we call the \textit{standard permutation categorical action} associated to the group $G\le S_{n}$. We summarize the above discussion in the following proposition.

\begin{prop}\label{o3vanish} Let $G\le S_{n}$, and $\rho_{n}:G\rightarrow EqBr(\cC^{\boxtimes n})$ be the canonical permutation homomorphism. Then the obstruction $o_{3}(\rho_{n})\in H^{3}(G, \operatorname{Inv}(\cC))$ vanishes. Furthermore, there is a canonical lift $\underline{\rho}_{n}: \underline{G}\rightarrow \underline{EqBr}(\cC^{\boxtimes n})\cong \underline{Pic}(\cC^{\boxtimes n})$ which we call the \textit{standard permutation categorical action}, constructed above.

\end{prop}

Note that we use $\underline{\rho}_{n}:\underline{S_{n}}\rightarrow \underline{EqBr}(\cC^{\boxtimes n})\cong \underline{Pic}(\cC^{\boxtimes n})$ to simultaneously refer to the monoidal functor whose images are in $\underline{EqBr}(\cC^{\boxtimes n})$ with $\underline{Pic}(\cC^{\boxtimes n})$, where we use the canonical monoidal equivalence $F$ from \eqref{F} identifying these two.

\medskip

Before we prove Theorem \ref{MainThm}, we wish to discuss other categorical actions associated to permuatation actions. As described in Section \ref{firstext}, given a categorical action $\underline{\rho}: \underline{G}\rightarrow \underline{EqBr}(\cC)\cong \underline{Pic}(\cC)$, all other categorical actions which truncate to the same homomorphism of groups $\rho:G\rightarrow EqBr(\cC)\cong Pic(\cC)$ can be obtained by twisting the original $\underline{\rho}$ by a 2-cocycle $\omega\in Z^{2}(G, \operatorname{Inv}(\cC))$. Here the action of $G$ on $\operatorname{Inv}(\cC)$ is induced by the homomorphism $\rho$. Equivalence classes of such twists only depend on the cohomology class in $H^{2}(G,\operatorname{Inv}(\cC))$.

In particular, for $G\le S_{n}$ and $\underline{\rho}_{n}:\underline{G}\rightarrow \underline{EqBr}(\cC^{\boxtimes n})$, there are other possible ``non-standard" permutation categorical actions, parametrized by the group $H^{2}(G,\operatorname{Inv}(\cC^{\boxtimes n}))$. We call these \textit{twisted} permutation categorical actions. We would like to determine necessary and sufficient conditions for there to be no twistings, at least when $G=S_{n}$. We know from \cite{FusionPermutation} that there is a unique permutation categorical action of $S_{2}$ on $\cC\boxtimes \cC$ for \textit{any} non-degenerate $\cC$. We have the following proposition, which gives us a condition in general.

\begin{prop} $H^{2}(S_{3}, \operatorname{Inv}(\cC^{\boxtimes 3}) )\cong \operatorname{Inv}(\cC)/2\operatorname{Inv}(\cC)$ and for all $n\ge 4$, $H^{2}(S_{n}, \operatorname{Inv}(\cC^{\boxtimes n}) )\cong (\operatorname{Inv}(\cC)/2\operatorname{Inv}(\cC))\times \operatorname{Inv}(\cC)_{2}$, where $\operatorname{Inv}(\cC)_{2}$ denotes the elements of order 2.
\end{prop}

\begin{proof} Consider the $S_{n-1}$ module $M:=\operatorname{Inv}(\cC)$, where $S_{n-1}$ acts trivially. Then clearly $Coind^{S_{n}}_{S_{n-1}}(M)\cong \operatorname{Inv}(\cC)^{n}$ with the ordinary permutation action. Thus by Shapiro's lemma, $H^{2}(S_{n-1}, M)\cong H^{2}(S_{n},\operatorname{Inv}(\cC)^{n})=H^{2}(S_{n}, \operatorname{Inv}(\cC^{\boxtimes n}))$, and thus we've reduced the problem to computing the second cohomology of $H^{2}(S_{n-1}, M)$ with trivial action on the coefficient group $M$. By the dual universal coefficient theorem, we have a short exact sequence 

$$0\rightarrow \operatorname{Ext}^{1}(H_{1}(S_{n-1}, \Z), M)\rightarrow H^{2}(S_{n-1}, M)\rightarrow \operatorname{Hom}(H_{2}(S_{n-1}, \Z), M)\rightarrow 0$$ 

\noindent which splits (non-naturally). Here the homology groups with coefficients in the integers have trivial action. But the first homology groups for $S_{n}$ are isomorphic to $\Z_{2}$ for all $n\ge 2$, while $H_{2}(S_{2}, \Z)$ and $H_{2}(S_{3}, \Z)$ are trivial and $H_{2}(S_{n}, \Z)\cong \Z_{2}$ for all $n\ge 4$.

First we consider the case $n=3$. We see that $\operatorname{Ext}^{1}(H_{1}(S_{2}, \Z), M)\cong H^{2}(S_{2}, M)\cong H^{2}(S_{3},\operatorname{Inv}(\cC^{\boxtimes n}))$. But $H_{1}(S_{2}, \Z)\cong \Z_{2}$, and $\operatorname{Ext}^{1}(\Z_{2}, M)\cong M/2M$.

Now suppose $n\ge 4$. Then $H_{1}(S_{n}, \Z)\cong H_{2}(S_{n}, \Z)\cong \Z_{2}$. The first non-zero term in the short exact sequence gives $\operatorname{Inv}(\cC)/2\operatorname{Inv}(\cC)$, while the second gives $\operatorname{Inv}(\cC)_{2}$.

\end{proof}

\begin{cor} For $n\ge 3$, the standard permutation categorical action $\underline{\rho}: \underline{S_{n}}\rightarrow \underline{Pic}(\cC^{\boxtimes{n}})$ is the unique permutation categorical action if and only if $|\operatorname{Inv}(\cC)|$ is odd.
\end{cor}

\begin{proof}
If $|\operatorname{Inv}(\cC)|$ is even, then multiplication by $2$ has a kernel, and thus $2\operatorname{Inv}(\cC)$ is a proper subgroup, hence $\operatorname{Inv}(\cC)/2\operatorname{Inv}(\cC)\ne 0$. Conversely, if $|\operatorname{Inv}(\cC)|$ is odd, then no element has even order so $\operatorname{Inv}(\cC)_{2}=0$. But then multiplication by $2$ is a bijection so $Inv(\cC)/2 \operatorname{Inv}(\cC)=0$.
\end{proof}

We would also very much like to know for which twistings the $o_{4}$ obstruction vanishes.

\begin{quest}\label{question} Let $G\le S_{n}$, $\cC$ be non-degenerate braided, and $\underline{\rho}_{n}:\underline{G}\rightarrow \underline{EqBr}(\cC^{\boxtimes n})$ be the standard permutation categorical action. For $\omega\in H^{2}(G, \operatorname{Inv}(\cC))$, let $\underline{\omega \triangleright \rho}$ denote the corresponding twisted permutation categorical action. What are necessary and sufficient conditions on $G$, $\omega$ and $\cC$ so that $o_{4}(\underline{\omega \triangleright \rho})$ vanishes?
\end{quest}

There is a concrete formula for the $o_{4}$ obstruction of the twisted permutation actions in terms of $\underline{\rho}$ and $\omega\in H^{2}(G, \operatorname{Inv}(\cC))$ given in Proposition 9 of \cite{MR3555361} that will likely be of use in answering this question. 

\medskip

We now turn to a proof of Theorem \ref{MainThm}.

\begin{proof}[Proof of Theorem \ref{MainThm}]

Let $\cC$ be a non-degenerate braided fusion category. Let $\underline{\rho}_{n}:\underline{S}_{n}\rightarrow \underline{EqBr}(\cC^{\boxtimes n})\cong \underline{Pic}(\cC^{\boxtimes n})$ be the standard permutation categorical action. We claim it suffices to show $o_{4}(\underline{\rho}_{n})$ vanishes for all $n\ge 8$. Indeed any $S_{n}$ can be embedded in $S_{m}$ for some $m\ge 8$, simply by considering its action on the set $\{1,2,\dots , n\}\subseteq \{1,\dots, m\}$ and we can view the restriction of $\underline{\rho}_{m}$ to $\underline{S}_{n}$ as the product action of $\underline{\rho}_{n}\times \underline{1}_{\cC^{\boxtimes (m-n)}}$. Hence by Lemma \ref{trivialproduct}, if $o_{4}(\underline{\rho}_{m})$ vanishes for all $m\ge 8$, then it vanishes for all n. However, Lemma \ref{trivialproduct} shows $o_{4}(\underline{\rho}_{8})$ vanishes if and only if $Res^{S_{m}}_{S_{8}}o_{4}(\underline{\rho}_{m})=o_{4}(\underline{\rho}_{8}\times \underline{1}_{\cC^{\boxtimes (m-8)}})$ does. However, by Nakaoka Stability (Corollary \ref{StableCor}), the restriction map $Res^{S_{m}}_{S_{8}}: H^{4}(S_{m}, \C^{\times})\rightarrow H^{4}(S_{8}, \C^{\times})$ is an isomorphism for all $m\ge 8$, proving the claim. 

Now we consider $S_{4}$. By Lemma \ref{product}, since the standard permutation categorical action of $S_{2}\times S_{2}\le S_{4}$ is a product action $\underline{\rho}_{2}\times \underline{\rho}_2$ and both these component actions have vanishing obstructions (since $H^{4}(\Z_{2}, \C^{\times})$ is trivial), we have from Equation \eqref{restriction} and Lemma \ref{product} that $Res^{S_{4}}_{S\times S_{2}}o_{4}(\underline{\rho}_{4})=o_{4}(\underline{\rho}_{2}\times \underline{\rho}_{2})$ vanishes. By Lemma \ref{S4}, $Res^{S_{4}}_{S_{2}\times S_{2}}$ is injective on $H^{4}(\cdot, \C^{\times})$, and thus $o_{4}(\underline{\rho}_{4})$ also vanishes. Applying a similar argument, we see that $Res^{S_{8}}_{S_{4}\times S_{4}}o_{4}(\underline{\rho}_{8})=o_{4}(\underline{\rho}_{4}\times \underline{\rho}_{4})$ vanishes. By Lemma \ref{S8}, $Res^{S_{8}}_{S_{4}\times S_{4}}$ is injective on $H^{4}(\cdot, \C^{\times})$. Thus $o_{4}(\underline{\rho}_{8})$ vanishes, concluding the argument.

\end{proof}

\section{Appendix: GAP computations.}\label{Appendix} These GAP computations require the package HAP. The following command computes the restriction map associated to the inclusion $H=\Z_{2}\times \Z_{2}\subseteq S_{4}=G$.

\medskip

G:=Group((1,2),(2,3),(3,4));

H:=Group((1,2),(3,4));

f:=GroupHomomorphismByImages(H,G,[(1,2),(3,4)],[(1,2),(3,4)]);

R:=ResolutionFiniteGroup(H,6);

S:=ResolutionFiniteGroup(G,6);

Amap:=EquivariantChainMap(R,S,f);

Hf:=Cohomology(HomToIntegers(Amap),5);

\medskip

Print(GroupCohomology(G,5));

Print(AbelianInvariants(Image(Hf)));

\medskip

\noindent The last two lines compute $H^{4}(S_{4}, \C^{\times})\cong H^{5}(G, \Z)=\Z_{2}$, and show that the image of the restriction map to $H^{4}(\Z_{2}\times \Z_{2}, \C^{\times})\cong H^{5}(H,\Z)$ is $\Z_{2}$, and in particular is injective. The following commands compute the restriction maps associated to the inclusion $L\cong (\Z_{2}\wr \Z_{2})\times (\Z_{2}\wr \Z_{2})\subseteq \langle\Z_{2}\wr \Z_{2}\wr \Z_{2},(35)(46)\rangle = K $.

\medskip

K:=Group((1,2),(3,4),(5,6), (7,8), (1,3)(2,4), (5,7)(6,8), (1,5)(2,6)(3,7)(4,8), (3,5)(4,6));

L:=Group((1,2),(3,4), (5,6), (7,8), (1,3)(2,4), (5,7)(6,8));

g:=GroupHomomorphismByImages(L,K,[(1,2),(3,4), (5,6), (7,8), (1,3)(2,4), (5,7)(6,8)];

[(1,2),(3,4), (5,6), (7,8), (1,3)(2,4), (5,7)(6,8)]);

T:=ResolutionFiniteGroup(L,6);

U:=ResolutionFiniteGroup(K,6);

Gmap:=EquivariantChainMap(T,U,g);

Hg:=Cohomology(HomToIntegers(Gmap),5);

\medskip

Print(GroupCohomology(K,5));

Print(AbelianInvariants(Image(Hg)));

\medskip

\noindent The last two lines compute $H^{4}(\langle\Z_{2}\wr \Z_{2}\wr \Z_{2},(35)(46)\rangle, \C^{\times})\cong H^{5}(K, \Z)=(\Z_{2})^{\times 9}$, and show that the image of the restriction map to $H^{4}( (\Z_{2}\wr \Z_{2})\times (\Z_{2}\wr \Z_{2}), \C^{\times})\cong H^{5}(L,\Z)$ is $(\Z_{2})^{\times 9}$, and in particular the restriction map is injective on $H^{4}$.

\bibliography{bibliography}
\bibliographystyle{plain}

\end{document}